\documentclass[preprint]{elsarticle}

\usepackage{amsmath}
\usepackage{amsthm}
\usepackage{amsfonts}
\usepackage{amssymb}
\usepackage{amsthm}

\usepackage{amsrefs}
\usepackage{hyperref}

\newtheorem{thm}{Theorem}
\newtheorem{prop}{Proposition}

\newtheorem{cor}{Corollary}
\newtheorem*{cor1}{Corollary 1}

\numberwithin{equation}{section}

\begin{document}

\title{An Euler operator approach to Ehrhart series}
\author{Wayne A. Johnson}
\address{Department of Mathematics\\
 Truman State University}
\ead{wjohnson@truman.edu}

\begin{abstract}
We use the ordinary Euler operator to compute the Ehrhart series for an arbitrary lattice polytope. The resulting formula involves the coefficients of the Ehrhart polynomial, combined via Eulerian numbers. We use this to compute $h^*_{d-1}$ in terms of the coefficients of the Ehrhart polynomial, resulting in a new linear inequality satisfied by the coefficents of the Ehrhart polynomial.
\vspace{1 pc}
\end{abstract}

\maketitle

\section{Introduction}

Let $\mathcal{P}$ be a $d$-dimensional, convex lattice polytope. For simplicity, assume that the vertices of $\mathcal{P}$ lie in the integer lattice in $\mathbb{R}^d$. Let $L_{\mathcal{P}}(t)$ be the number of integer lattice points in the $t$-fold dilation of $\mathcal{P}$. Eug\`{e}ne Ehrhart proved in \cite{EE} that $L_\mathcal{P}(t)$ is a polynomial in $t$. This polynomial is referred to as the Ehrhart polynomial of $\mathcal{P}$.

Let $L_\mathcal{P}(t)=L_0+L_1t+\dots+L_dt^d$. The coefficients and roots of $L_\mathcal{P}(t)$ have been studied extensively. In \cite{BDDPS}, the authors present a list of inequalities satisfied by the coefficients of $L_\mathcal{P}$ and explore its roots. One of the goals of this study was to classify the possible Ehrhart polynomials, at least in low dimension. This is discussed in Sections 2 and 3 of \cite{BDDPS}.

One way of computing the Ehrhart polynomial of $\mathcal{P}$ is to study its Ehrhart series. The Ehrhart series of $\mathcal{P}$ is the generating function of the values of $L_\mathcal{P}(t)$ at non-negative integers. In other words,
\begin{center}
$\text{Ehr}_\mathcal{P}(z):=\displaystyle\sum_{t\geq0}L_\mathcal{P}(t)z^t$.
\end{center}
The constant term $L_0$ is known to be 1 (see, for example, Chapter 3 of \cite{BR}). Ehrhart himself showed that $\text{Ehr}_\mathcal{P}(z)$ converges to a rational function of the form
\begin{center}
$\displaystyle\frac{h^*_0+h^*_1z+\dots+h^*_dz^d}{(1-z)^{d+1}}$.
\end{center}
In \cite{S1}, Stanley showed (in his celebrated non-negativity theorem), that the coefficients in the numerator are non-negative integers. The vector, $(h^*_0,\dots,h^*_d)$, formed by these integers is called the $h^*$-vector (or sometimes the $\delta$-vector) of $\mathcal{P}$ for its relationship with the standard $h$-vector of $\mathcal{P}$. Details of this relationship can be found in Chapter 10 of \cite{BDDPS}.

Of note in the present article is the computation of $\text{Ehr}_\mathcal{P}(z)$ when $\mathcal{P}$ is the unit $d$-dimensional cube, $[0,1]^d$ (see Chapter 2 of \cite{BDDPS} or Section 4.6 of \cite{S2}). In this case, we have
\begin{center}
$L_\mathcal{P}(t)=(t+1)^d$,
\end{center}
and thus,
\begin{center}
$\text{Ehr}_\mathcal{P}(z)=\displaystyle\sum_{t\geq0}(t+1)^dz^t$.
\end{center}
The series above is computed using the fact that
\begin{center}
$\displaystyle\sum_{j\geq0}j^dz^j=\left(z\frac{d}{dz}\right)^d\displaystyle\frac{1}{1-z}$
\end{center}
The differential operator $z\displaystyle\frac{d}{dz}$ is often called the Euler operator for its relationship with the Eulerian numbers. This relationship is discussed in more detail in the next section.

This paper takes the perspective of generalizing the computation for the unit cube to any lattice polytope. This can be done because the Euler operator can be used to compute a rational form for any power series whose coefficients are polynomial. The details of this are in the proof of Theorem 1 below. Here, we focus on the Corollary to Theorem 1.

\begin{cor1}
If $\mathcal{P}$ is a convex, $d$-dimensional lattice polytope, then we have
\begin{center}
$\text{Ehr}_{\mathcal{P}}(z)=\displaystyle\frac{1}{1-z}+\displaystyle\sum_{i=1}^d L_i\frac{\phi_i(z)}{(1-z)^{i+1}}$,
\end{center}
where $\phi_i(z)$ is the $i^\text{th}$ Euler polynomial.
\end{cor1}

This form of the Ehrhart series is useful for two reasons. Firstly, we can use this to compute the numerator (and thus the $h^*$-vector of $\mathcal{P}$) in terms of the coefficients of the Ehrhart polynomial recursively. Secondly, this computation (in combination with the fact that the terms in $h^*$ are non-negative) can be used to compute linear inequalities that the coefficients of the Ehrhart polynomial must satisfy. In \S3, we pursue both of these lines. We give a formula for a new linear inequality that the coefficients of the Ehrhart polynomial satisfy (in Theorem 2), and explore the recursive relationship for low-dimensional polytopes. The main content of \S2, below, is a proof of the main theorem and the above corollary.

\section{Main Theorem}

In this section, we present a proof of the main theorem, and its main corollary. The proof is elementary, relying only on the fact that the coefficients of $Ehr(z)$ are polynomial in $t$.

\begin{thm}
Let $\mathcal{P}$ be a lattice polytope with Ehrhart series $\text{Ehr}_{\mathcal{P}}(z)$. Then
\begin{center}
$\text{Ehr}_{\mathcal{P}}(z)=\left(1+L_1z\displaystyle\frac{d}{dz}+\dots+L_d\left(z\frac{d}{dz}\right)^d\right)\displaystyle\frac{1}{1-z}$.
\end{center}
\end{thm}

\begin{proof}
First, note that
\begin{center}
$\text{Ehr}_{\mathcal{P}}(z)=\displaystyle\sum_{t\geq0}(1+L_1t+\dots+L_dt^d)z^t$\\
$=\displaystyle\sum_{t\geq0}(z^t+L_1tz^t+\dots+L_dt^dz^t)$\\
$=\displaystyle\sum_{t\geq0}z^t+L_1\sum_{t\geq0}tz^t+\dots+L_d\sum_{t\geq0}t^dz^t$.
\end{center}
To finish the proof, we compute $f_i(z):=\displaystyle\sum_{t\geq0}t^iz^t$ for arbitrary $i$. This computation is not new, but the details will be used in the sequel. Note that $f_0(z)=\displaystyle\sum_{t\geq0}z^t$. If we apply the differential operator $z\displaystyle\frac{d}{dz}$ to $f_i(z)$, we get
\begin{center}
$z\displaystyle\frac{d}{dz}f_i(z)=\sum_{t\geq0}t^iz\frac{d}{dz}z^t=\sum_{t\geq0}t^iztz^{t-1}$\\
$=\displaystyle\sum_{t\geq0}t^{i+1}z^t=f_{i+1}(z)$.
\end{center}
Therefore, 
\begin{center}
$f_{i}(z)=\displaystyle\left(z\frac{d}{dz}\right)^if_0(z)=\left(z\frac{d}{dz}\right)^i\frac{1}{1-z}$,
\end{center}
Putting it all together, we have
\begin{center}
$\text{Ehr}_{\mathcal{P}}(z)=\displaystyle\frac{1}{1-z}+L_1\left(z\frac{d}{dz}\right)\frac{1}{1-z}+\dots+L_d\left(z\frac{d}{dz}\right)^d\frac{1}{1-z}$\\
$=\left(1+L_1z\displaystyle\frac{d}{dz}+\dots+L_d\left(z\frac{d}{dz}\right)^d\right)\displaystyle\frac{1}{1-z}$.
\end{center}
\end{proof}

The terms in the Ehrhart series of the form $\left(z\displaystyle\frac{d}{dz}\right)^i\displaystyle\frac{1}{1-z}$ have an interesting and well-studied structure. In particular,
\begin{center}
$\left(z\displaystyle\frac{d}{dz}\right)^i\displaystyle\frac{1}{1-z}=\displaystyle\frac{\phi_i(z)}{(1-z)^{i+1}}$,
\end{center}
where $\phi_i(z)$ is the $i^\text{th}$ Euler polynomial. We then have an immediate corollary to the theorem.

\begin{cor}
Under the assumptions of the theorem, we have
\begin{center}
$\text{Ehr}_{\mathcal{P}}(z)=\displaystyle\frac{1}{1-z}+\displaystyle\sum_{i=1}^d L_i\frac{\phi_i(z)}{(1-z)^{i+1}}$.
\end{center}
\end{cor}

It is this form of the Ehrhart series that we wish to study further. Note that the coefficients of the Euler polynomial $\phi_i(z)$ are called \textit{Eulerian numbers} and have been studied extensively. In particular, if we set
$\phi_i(z)=\displaystyle\sum_{j=0}^i a_{i,j}z^j$, then the Euler number $a_{i,j}$ is the number of permutations of $i$ with exactly $j-1$ rises (see \cite{EN} for more details and further combinatorial interpretations). Note that $a_{i,0}=0$ for all $i>0$.

\section{Computation of $h^*$-components}

In this section, we use the formula in Corollary 1 to compute the components of the $h^*$-vector. The content of this section has two flavors: (1) What can we say in general about a fixed component, $h^*_k$? and (2) Can we use the $h^*$-vector in low dimension to develop new constraints on the coefficients of $L_\mathcal{P}(t)$? The former is answered for $h_1^*$, $h_{d-1}^*$, and $h_d^*$. Only the formula for $h_{d-1}^*$ is new. The latter is answered in dimensions 3 and 4. 

Propositions 1 and 2 below are not new, but their proofs are. The proofs could be left as exercises, but we present them, as they illustrate how simple it is to deal with the form of the Ehrhart series in Corollary 1, as long as the Eulerian numbers involved are well-behaved. Propositions 3 and 4 provide a computation of $h^*_{d-1}$ in terms of the coefficients of $L_\mathcal{P}(t)$. These are then used to present a new linear inequality on these coefficients in Theorem 2.

\begin{prop}
Under the assumptions of the theorem, we have
\begin{center}
$h_1^*=L_1+\dots+L_d-d$.
\end{center}
\end{prop}

\begin{proof}
We prove this by induction on $d$. If $d=1$, then we have
\begin{center}
$\text{Ehr}_{\mathcal{P}}(z)=\displaystyle\frac{1}{1-z}+\frac{L_1\phi_1(z)}{(1-z)^2}$\\
$=\displaystyle\frac{1-z+L_1z}{(1-z)^2}=\frac{1+(L_1-1)z}{(1-z)^2}$.
\end{center}
Now, assume the proposition holds for $d=k-1$. Then, if $d=k$,  we have
\begin{center}
$\text{Ehr}_{\mathcal{P}}(z)=\displaystyle\frac{1+(L_1+\dots+L_{k-1}-(k-1))z+\text{H.O.T.}}{(1-z)^{k}}+\frac{L_k\phi_k(z)}{(1-z)^{k+1}}$
\end{center}
Note that H.O.T. stands for ``higher order terms" which do not contribute to $h_1^*$. Multiplying the fraction on the left by $(1-z)$ on top and bottom and adding fractions yields
\begin{center}
$\displaystyle\frac{1+(L_1+\dots+L_{k-1}-(k-1))z-z+\text{H.O.T.}+L_kz+\text{H.O.T.}}{(1-z)^{k+1}}$
\end{center}
Combining like terms gives $L_1+\dots+L_k-k$ as the coefficient of $z$.
\end{proof}

Note that, since $h_1^*$ is a positive integer, it follows immediately that 
\begin{center}
$L_1+\dots+L_d\geq d$.
\end{center}
This is inequality (7) in Theorem 3.5 in \cite{BDDPS}.

The above proof is made simple by the fact that the lowest order term of $\phi_d(z)$ is $z$. Similarly, the next proof relies on the fact that the highest order term of $\phi_d(z)$ is $z^d$.

\begin{prop}
Under the assumptions of the theorem, we have
\begin{center}
$h_d^*=\displaystyle\sum_{i=0}^d(-1)^{d+i}L_i$.
\end{center}
\end{prop}

\begin{proof}
We prove this by induction on $d$. If $d=1$, we have $h_1^*=-1+L_1$, as above. Now, assume the proposition holds for $d=k-1$. Then, if $d=k$, the Ehrhart series of $\mathcal{P}$ has the form
\begin{center}
$\displaystyle\frac{\text{L.O.T}+\left((-1)^{k-1}+\displaystyle\sum_{i=1}^{k-1}(-1)^{k-1+i}L_i\right)z^{k-1}}{(1-z)^k}+\displaystyle\frac{L_k\phi_k(z)}{(1-z)^{k+1}}$
\end{center}
Note that L.O.T. stands for ``lower order terms" which do not contribute to $h_d^*$. We once again get a common denominator by multiplying the top and bottom of the leftmost fraction by $(1-z)$. This yields the following (for simplicity, we only write the numerator):
\begin{center}
$\text{L.O.T}-z\left((-1)^{k-1}+\displaystyle\sum_{i=1}^{k-1}(-1)^{k-1+i}L_i\right)z^{k-1}+\text{L.O.T.}+L_kz^k$.
\end{center}
Simplifying the coefficient of $z^k$ gives
\begin{center}
$(-1)^k+\displaystyle\sum_{i=1}^{k-1}(-1)^{k+i}L_i + L_k=(-1)^{k}+\displaystyle\sum_{i=1}^k(-1)^{k+i}L_i$,
\end{center}
as desired.
\end{proof}

As with $h^*_1$, since $h^*_d\geq0$, the above theorem implies the inequality
\begin{center}
$\displaystyle\sum_{i=0}^d(-1)^{d+i}L_i\geq0$,
\end{center}
which is inequality (11) in \cite{BDDPS}.

As a next example, let $d=3$. Since $h_0^*=1$, and we can compute $h_1^*$ and $h_3^*$ from the propositions, to compute the $h^*$-vector, we need only compute $h_2^*$. This is not hard to do, and it yields
\begin{center}
$h^*=(1,L_1+L_2+L_3-3, -2L_1+4L_3+3, L_1-L_2+L_3-1)$.
\end{center}
We can think of this as three linear equations in three unknowns. Namely,
\begin{center}
$h_1^*=L_1+L_2+L_3-3$\\
$h_2^*=-2L_1+4L_3+3$\\
$h_3^*=L_1-L_2+L_3-1$,
\end{center}
where the unknowns are $L_1$, $L_2$, and $L_3$. Solving this system yields formulas for the coefficients of the non-constant terms of the Ehrhart polynomial of $\mathcal{P}$:
\begin{center}
$L_1=\displaystyle\frac{2h_1^*-h_2^*+2h_3^*+11}{6}$\\
$L_2=\displaystyle\frac{h_1^*-h_3^*+2}{2}$\\
$L_3=\displaystyle\frac{h_1^*+h_2^*+h_3^*+1}{6}$.
\end{center}
Thus, for a polynomial of degree 3 to be an Ehrhart polynomial, it must satisfy these three relations. In particular, the Ehrhart polynomial of a polytope of dimension 3 must be of the form
\begin{center}
$1+\displaystyle\frac{2h_1^*-h_2^*+2h_3^*+11}{6}t+\displaystyle\frac{h_1^*-h_3^*+2}{2}t^2+\displaystyle\frac{h_1^*+h_2^*+h_3^*+1}{6}t^3$.
\end{center}

We now turn to the problem of computing $h_{d-1}^*$. To do this, we define some notation. If $\mathcal{P}$ is $d$-dimensional, we set
\begin{center}
$_{d-1}h^*_{d-1}:=\displaystyle\sum_{i=0}^{d-1}(-1)^{d-1+i}L_i$.
\end{center}
This is the formula from Proposition 2 for a $d-1$-dimensional polytope, but the coefficients are from the Ehrhart polynomial of a $\mathcal{P}$. We define $_{d-1}h^*_{d-2}$ in the same way. Then, general $d$, $h^*_{d-1}$ can be found recursively by computing
\begin{center}
$_{d-1}h^*_{d-1}$ $-$ $_{d-1}h^*_{d-2}+a_{d,d-1}L_d$,
\end{center}
as this is the coefficient of $z^{d-1}$ in the numerator of $\text{Ehr}_\mathcal{P}(z)$, expanded using Corollary 1. We now focus on this computation. Note that two of the terms are not complicated. We have
\begin{center}
$_{d-1}h^*_{d-1}=\displaystyle\sum_{i=0}^{d-1}(-1)^{d-1+i}L_i$,
\end{center}
by Proposition 2, and it is known that $a_{d,d-1}=2^d-d-1$ (see \cite{EN1}). We need only compute $_{d-1}h^*_{d-2}$ to compute $h^*_{d-1}$, which we do recursively for $d\geq3$. These numbers are given below in low dimension. For $i\geq d$, if $L_i$ does not occur in the formula for $h_{d-1}^*$, we still include it with a coefficient of 0. This is to make the patterns as $d$ gets larger more obvious.

\begin{center}
$\begin{array}{l|c}
d & h^*_{d-1}\\
 & \\ 
\hline
 & \\ 
1 & 1+0L_1\\
 & \\  
2 & -2+L_1+L_2\\
 & \\ 
3 & 3-2L_1-0L_2+4L_3\\
 & \\ 
4 & -4+3L_1-L_2-3L_3+11L_4\\
 & \\ 
5 & 5-4L_1+2L_2+2L_3-10L_4+26L_5\\
 & \\ 
6 & -6+5L_1-3L_2-L_3+9L_4-25L_5+57L_6\\
 & \\ 
7 & 7-6L_1+4L_2+0L_3-8L_4+24L_5-56L_6+120L_7\\
 & \\ 
8 & -8+7L_1-5L_2+L_3+7L_4+23L_5-55L_6-119L_7+247L_8
\end{array}$
\end{center}

Recall that $a_{1,0}=0$. For general $d$, the constant term in the above table is $(-1)^{d+1}d$. The coefficient of $L_d$ is $a_{d,d-1}$. The remaining coefficients follow an interesting pattern. To see this, follow $L_1$ through the table. The coefficient of $L_1$ is $(-1)^d(d-1)$. More interestingly, follow $L_2$ through the table. The coefficients are $1,0,-1,2,-3,4,\dots$. The first coefficient on the list is $a_{2,1}$. The next coefficient is $a_{2,1}-1=1-1=0$. We will show in Proposition 4 below that the general pattern is
\begin{center}
$(-1)^d(a_{2,1}-(d-2))$.
\end{center}

We next follow $L_3$ through the table. The coefficients of $L_3$ are 
\begin{center}
$4,-3,2,-1,0,1,-2,3,\dots$. 
\end{center}
The first coefficient is $a_{3,2}=4$. The next coefficient is $(-1)(a_{3,2}-1)=-3$. The next coefficient is $(-1)^2(a_{3,2}-2)=2$, and the next is $(-1)^3(a_{3,2}-3)=-1$. From then on out, the coefficient is of the form $(-1)^{d+1}(d-7)$. In a different form, we have that the coefficient of $L_3$ is the following for $d\geq3$:
\begin{center}
$(-1)^{d+1}(a_{3,2}-(d-3))$.
\end{center}

The difference between the $L_2$ and the $L_3$ case is that, after getting to zero, the coefficients of $L_2$ start negative, while the coefficients of $L_3$ start positive. This occurs because the sum $d+a_{d,d-1}=d+2^d-d-1=2^d-1$ is always odd. Therefore, the coefficient of $L_k$ is zero in an odd row. If $k$ is even, the next step will be $-1$, and, if $k$ is odd, the next step will be $+1$. We deal with these cases separately.

\begin{prop}
Let $k$ be odd. Then, for $d\geq k$, the coefficient of $L_k$ in $h^*_{d-1}$ is
\begin{center}
$(-1)^{d+1}(a_{k,k-1}-(d-k))=(-1)^{d+1}(2^k-1-d)$.
\end{center}
\end{prop}

\begin{proof}
We proceed by induction on $d$. If $d=k$, then we have
\begin{center}
$(-1)^{d+1}(a_{k,k-1}-(d-d))=a_{k,k-1}$.
\end{center}
Now, assume the proposition holds for $d=n-1>k$. So, in row $n-1$, the coefficient of $L_k$ is 
\begin{center}
$(-1)^{n}(a_{k,k-1}-(n-1-d))$. 
\end{center}
Let $d=n$. Then, ignoring $L_n$, $h_{n-1}^*$ is formed by taking the $n-1^{\text{st}}$ row of the table and subtracting it from $\displaystyle\sum_{i=0}^{n-1}(-1)^{n-1+i}L_i$. The coefficient of $L_k$ in the sum is $(-1)^{n-1+k}$. So, the coefficient of $L_k$, when $d=n$, is
\begin{center}
$(-1)^{n-1+k}-(-1)^{n}(a_{k,k-1}-(n-1-d))$.
\end{center}
If $n$ is odd, then $n-1+k$ is also odd, so we have
\begin{center}
$-1+(a_{k,k-1}-(n-1-d))=a_{k,k-1}-(n-d)=(-1)^{n+1}a_{k,k-1}-(n-d)$,
\end{center}
as desired. If $n$ is even, then $n-1+k$ is also even. So we have
\begin{center}
$1-(a_{k,k-1}-(n-1-d))=-a_{k,k-1}+(n-d)=(-1)^{n+1}(a_{k,k-1}-(n-d))$.
\end{center}
In either case, the claim is proved.
\end{proof}

\begin{prop}
Let $k$ be even. Then for $d\geq k$, the coefficient of $L_k$ in $h^*_{d-1}$ is
\begin{center}
$(-1)^d(a_{k,k-1}-(d-k))=(-1)^d(2^k-1-d)$.
\end{center}
\end{prop}

\begin{proof}
We again proceed by induction on $d$. If $d=k$, we have
\begin{center}
$(-1)^k(a_{k,k-1}-(k-k))=a_{k,k-1}$,
\end{center}
since $(-1)^k=1$, as $k$ is even. Next, assume the claim holds when $d=n-1$. So, in row $n-1$ of the table, the coefficient of $L_k$ is
\begin{center}
$(-1)^{n-1}(a_{k,k-1}-(n-1-d))$.
\end{center}
Let $d=n$. Then, ignoring $L_n$, $h^*_{n-1}$ is formed by taking the $n-1^{\text{st}}$ row of the table and subtracting it from $\displaystyle\sum_{i=0}^{n-1}(-1)^{n-1+i}L_i$. The coefficient of $L_k$ in the sum is $(-1)^{n-1+k}$. So, the coefficient of $L_k$, when $d=n$, is
\begin{center}
$(-1)^{n-1+k}-(-1)^{n-1}(a_{k,k-1}-(n-1-d))$.
\end{center}
If $n$ is odd, then $n-1-k$ is even, so we have
\begin{center}
$1-(a_{k,k-1}-(n-1-d))=-(a_{k,k-1}-(n-d))=(-1)^n(a_{k,k-1}-(n-d))$.
\end{center}
If $n$ is even, then $n-1+k$ is odd, so we have
\begin{center}
$-1+(a_{k,k-1}-(n-1-d))=a_{k,k-1}-(n-d)=(-1)^n(a_{k,k-1}-(n-d))$.
\end{center}
In either case, the claim holds.
\end{proof}

Propositions 3 and 4, along with the constant term $(-1)^{d+1}d$ describe every coefficient of the expansion of $h^*_{d-1}$ in terms of the coefficients of the Ehrhart polynomial. Putting all of this together yields the following linear inequality.

\begin{thm}
Let $d$ be a positive integer. Let $c_0=(-1)^{d+1}d$, and, for $k\geq1$, let
\begin{center}
$c_k=\begin{cases}
			(-1)^{d+1}(2^k-1-d) & \text{, if }k\text{ is odd}\\
			(-1)^{d}(2^k-1-d) & \text{, if }k\text{ is even}
	 \end{cases}$
\end{center}
Then the Ehrhart polynomial of $\mathcal{P}$ satisfies the linear inequality
\begin{center}
$c_0L_0+c_1L_1+\dots+c_dL_d\geq0$.
\end{center}
\end{thm}

\begin{proof}
This follows immediately from Propositions 3 and 4, as well as the fact that $h_{d-1}^*$ is positive.
\end{proof}

We finish by computing the $h^*$-polynomial for a 4-dimensional polytope. The recursive nature of the formula in Corollary 3 makes this simple (if a bit tedious in higher dimension). Recall that the numerator of $\text{Ehr}(z)$ for a 3-dimensional polytope has the form
\begin{center}
$N(z)=1+(L_1+L_2+L_3-3)z+(-2L_1+4L_3+3)z^2 + (L_1-L_2+L_3-1)z^3$.
\end{center}
The numerator of $\text{Ehr}(z)$ for a 4-dimensional polytope is then
\begin{center}
$(1-z)N(z)+L_4\phi_4(z)$.
\end{center}
Note that $\phi_4(z)=z+11z^2+11z^3+z^4$. A simple computation yields
\begin{center}
$h^*_0=1$\\
$h^*_1=L_1+L_2+L_3+L_4-4$\\
$h^*_2=-3L_1-L_2+3L_3+11L_4+6$\\
$h^*_3=3L_1-L_2-3L_3+11L_4-4$\\
$h^*_4=1-L_1+L_2-L_3+L_4$.
\end{center}
The only value above that is not directly computable using the Propositions of this paper is $h^*_2$. This term provides the inequality
\begin{center}
$-3L_1-L_2+3L_3+11L_4\geq-6$
\end{center}
in addition to the inequalities given by the Propositions. This could be continued in two ways: (1) we could compute the $h^*$-vector (and obtain inequalities) in higher dimensions by computing $(1-z)N(z)+L_d\phi_d(z)$, where $N(z)$ is the numerator of the Ehrhart series in dimension $d$, and (2) we could compute the values of the $h^*$-vector for general $d$ when nice formulas for the Eulerian numbers are known. For the latter, it is known that
\begin{center}
$a_{d,d-2}=3^{d+2}-(d+3)2^{d+2}+\displaystyle\frac{(d+2)(d+3)}{2}$. (see \cite{EN2})
\end{center}
This could be used to compute $h_{d-2}^*$ in general $d$.

\end{document}